\documentclass[12pt]{amsart}

\usepackage{amsmath,amssymb,amsbsy,amsfonts,amsthm,latexsym,
            amsopn,amstext,amsxtra,euscript,amscd,stmaryrd,mathrsfs,
            cite,array,bm,soul}

\numberwithin{equation}{section}

\usepackage{color}
\usepackage[colorlinks,linkcolor=blue,anchorcolor=blue,citecolor=blue,backref=page]{hyperref}
\hypersetup{breaklinks=true}
\usepackage[norefs,nocites]{refcheck}

\newcommand{\mand}{\qquad\mbox{and}\qquad }
\newcommand{\fl}[1]{{\ensuremath{\left\lfloor #1 \right\rfloor}} }

\renewcommand{\(}{\ensuremath{\left(} }
\renewcommand{\)}{\ensuremath{\right)} }
\newcommand{\e}{\ensuremath{\mathbf{e}} }

\renewcommand{\le}{\ensuremath{\leqslant} }
\renewcommand{\ge}{\ensuremath{\geqslant} }
\renewcommand{\leq}{\ensuremath{\leqslant} }
\renewcommand{\geq}{\ensuremath{\geqslant} }

\newcommand{\ssum}[1]{\sum_{\substack{#1}}}  

\def\gdx{{G}}

\def\kdx{{k}}
\def\be{\begin{equation}}
\def\ee{\end{equation}}
\def\padic{\nu_q}

\def\ord{\mathop{\rm ord}\nolimits}

\newtheorem{X}{X}[section]
\newtheorem{theorem}[X]{Theorem}  
\newtheorem{lemma}[X]{Lemma}
\newtheorem{cor}[X]{Corollary}

\newtheorem{proposition}[X]{Proposition}

\newcommand{\cA}{\ensuremath{\mathcal{A}} }

\newcommand{\cH}{\ensuremath{\mathcal{H}} }
\newcommand{\cI}{\ensuremath{\mathcal{I}} }
\newcommand{\cJ}{\ensuremath{\mathcal{J}} }

\newcommand{\cM}{\ensuremath{\mathcal{M}} }
\newcommand{\cN}{\ensuremath{\mathcal{N}} }

\newcommand{\cS}{\ensuremath{\mathcal{S}} }

\newcommand{\cY}{\ensuremath{\mathcal{Y}} }
\newcommand{\cZ}{\ensuremath{\mathcal{Z}} }

\newcommand{\N}{\ensuremath{\mathbb{N}} }

\newcommand{\R}{\ensuremath{\mathbb{R}} }

\newcommand{\Z}{\ensuremath{\mathbb{Z}} }

\newcommand{\NN}{\ensuremath{\mathbb{N}} }

\usepackage{todonotes}

\def \balpha{\bm{\alpha}}
\def \bbeta{\bm{\beta}}
\def \bgamma{\bm{\gamma}}

\begin{document}

\title[On digits of  Mersenne numbers]
      {On digits of  Mersenne numbers}
      
%

      
\author[B.\ Kerr]{Bryce Kerr}
\address{Max Planck Institute for Mathematics, Bonn, Germany}
\email{bryce.kerr89@gmail.com}

\author[L.\ M\'erai]{L\'aszl\'o M\'erai}

\address{Johann Radon Institute for Computational and Applied Mathematics, 
Austrian Academy of Sciences,  
Altenberger Stra\ss e 69, A-4040 Linz, Austria}

\email{laszlo.merai@oeaw.ac.at}

\author[I.\ E.\ Shparlinski]{Igor E.\ Shparlinski} 

\address{Department of Pure Mathematics,
		 University of New South Wales,
		 Sydney, NSW 2052, Australia.}

\email{igor.shparlinski@unsw.edu.au}
        
\date{\today}

\begin{abstract}
Motivated by recently developed interest to the distribution of $q$-ary digits of Mersenne numbers $M_p = 2^p-1$, 
where $p$ is prime, we estimate 
rational exponential sums with $M_p$, $p \le X$, modulo a large power of a fixed odd prime $q$. In turn this immediately 
implies the normality of strings of $q$-ary digits amongst about $(\log X)^{3/2+o(1)}$ rightmost digits of   $M_p$, $p \le X$.
Previous results imply this only for about $(\log X)^{1+o(1)}$  rightmost  digits.
\end{abstract}


\maketitle


\section{Introduction}

\subsection{Overview} 
Recently,   Cai, Faust,   Hildebrand,   Li and Zhang~\cite{CFHLZ} have   considered 
various questions on the patterns in leading $q$-ary digits of {\it Mersenne numbers\/} $M_p = 2^p-1$, 
where $p$ is prime, see also~\cite{CHL,HHLZ} for some other related questions. 
 In particular, 
one can find in~\cite{CFHLZ}  some numerical results which suggest the leftmost  $q$-ary digits 
of  Mersenne numbers obey the so-called {\it  Benford law}.  It has also been observed in~\cite[Remark~4.4 and Section~7]{CFHLZ} 
that the bounds of exponential sums with fractions $M_p/m$ for a large integer $m$ such as in~\cite{BCFS, BFGS}
 can be used to extract some nontrivial information about the distribution of   the rightmost digits of $M_p$.
This conclusion in~\cite{CFHLZ}  is based on bounds of exponential sums with an arbitrary modulus $m$.
However for the case $q$-ary digits only moduli of the form $m = q^\gamma$ with an integer $\gamma$
are of interest. Here we show that indeed for such moduli, using some ideas of Korobov~\cite{Kor} 
one can obtain much stronger results. To emphasise the ideas we consider the case when $q$ is  prime,
however there is no doubt that the method extends to any $q$ without too much loss in its power. 

For example,  our bounds of exponential sums immediately imply the following equidistribution results 
for $q$-ary digits of $M_p$. 
For any fixed real $\varepsilon> 0$ 
and for any positive integers $s\leq r \le (\log  X)^{3/2-\varepsilon}$, on
rightmost 
$q$-ary positions  $r, \ldots, r-s+1$ of $M_p$, $p \le X$,  any block of $q$-ary digits of length $s$ 
appears  asymptotically  the same number of times, 
that is $\left(q^{-s} + o(1)\right)\pi(X)$, where, as usual, $\pi(X)$ 
denotes the number of primes $p \le X$, see Theorem~\ref{thm:MersDig}.

The generic results of~\cite{BCFS, BFGS} imply this only for positions which are much closer to 
the right end, 
namely, only for $r \le c \log  X$ for some absolute constant $c>0$. 

Let $m$ be an arbitrary natural number, and let $a$ and $g$ be integers that
are coprime to $m$.  In this paper, we study exponential sums of the form
\be\label{eq:SmaN}
S_m(a;X)=\sum_{n\le X}\Lambda(n)\e_m(ag^n),
\ee 
where $\e_m$ is the additive character modulo $m$ defined by
$$
\e_m(t)=\exp(2\pi i t/m)\qquad(t\in\R),
$$
and $\Lambda$ is the \emph{von Mangoldt function}:
$$
\Lambda(n) =
\begin{cases}
\log p&\quad\text{if $n$ is a power of the prime $p$,}\\
 0 &\quad\text{otherwise.}
\end{cases}
$$

The sums~\eqref{eq:SmaN} are introduced in Banks \textit{et al}~\cite{BCFS},
where it is shown that
$$
\max_{(a,m)=1}
|S_m(a;X)|\le\(X\tau^{-11/32}m^{5/16}
+X^{5/6}\tau^{5/48}m^{7/24}\)X^{o(1)} 
$$
as $X \to \infty$, 
where  $\tau = \ord_m g$ denotes the multiplicative order of $g$ modulo $m$,
that is, the smallest natural number $k$ such that $g^k\equiv 1\bmod m$.

Using an idea of Garaev~\cite{Gar2} to handle double sums over certain
hyperbolic regions, the stronger bound
$$
\max_{(a,m)=1}|S_m(a;X)|\le\(X\tau^{-11/32}m^{5/16}
+X^{4/5}\tau^{1/8}m^{7/20}\)X^{o(1)}
$$
is established in Banks \textit{et al}~\cite{BFGS}. Note that, for
either of the above bounds to be nontrivial, one must have $\tau\ge m^{10/11}X^{o(1)}$
(to control the first term), hence also 
$$
m\le X^{22/51+o(1)}
$$
(to control the second), as $X\to \infty$.  For shorter sums, new ideas are needed.

In the present paper, we study the exponential sums $S_m(a;X)$
in the special case that $m=q^\gamma$ for some fixed prime $q$. Our aim is to
establish nontrivial bounds for short sums in which $X$ is smaller than the modulus $m$.
Our approach relies on an idea of Korobov~\cite{Kor} coupled with the use of
Vinogradov's mean value theorem in the explicit form given by Ford~\cite{Ford}. 

\subsection{Statement of results}

Since our main motivation comes from applications to Mersenne numbers, we always assume that $q \ge 3$, which simplifies the formulas in Section~\ref{sec:tech-lemmas}
(and can easily be avoided at the cost of some small typographical changes).  

\begin{theorem}\label{thm:main}
Fix a prime $q\ge 3$ and an integer $g\ge 2$ not divisible by $q$. Let $\gamma$ be a positive integer, $A>0$ be an arbitrary constant and suppose $X\geq 2$ satisfies
\be
\label{eq:mainNcond}
 X\le q^{A\gamma}.
\ee 
Then, for all integers $a$ with $\gcd(a,q)=1$,  we have 
$$
\left|S_{q^{\gamma}}(a;X)\right|  \leq c(g,q, A) \(X^{1-\delta(A)  \rho^2}\log X + Xq^{-\delta(A) \gamma}\),
$$
where $\delta(A) > 0$ is a constant depending only on $A$,
\be
\label{eq:delta rho}
 \rho =\frac{\log X}{\log q^\gamma} 
\ee 
and $c(g,q,A)$ depends only on $g$, $q$ and $A$.
\end{theorem}

We remark that Theorem~\ref{thm:main}   is nontrivial in the range
$$
q^{A \gamma} \geq X \geq q^{\gamma^{2/3+\varepsilon}},
$$
for an arbitrary small $\varepsilon > 0$, provided that $X$ is large enough, and 
with $g=2$ yields (via partial summation) a nontrivial
bound on exponential sums with  Mersenne numbers  $M_p=2^p-1$, $p$ prime.  

\begin{cor}
\label{cor:Mers}  
For a prime $q\ge 3$ and a real $X\geq 2$  satisfying~\eqref{eq:mainNcond}
we have $$
\max_{(a,q)=1}\left|\ssum{p\le X\\p\text{~prime}}\e_{q^\gamma}(aM_p)\right|
\leq c(q, A) \(X^{1-\delta_0(A)  \rho^2}   + Xq^{-\delta_0(A) \gamma}\), 
$$
where $\delta_0(A)> 0$ is a constant depending only on $A$, $\rho$ is as in~\eqref{eq:delta rho}
and $c(q,A)$ depends only on $q$ and $A$.  
\end{cor}

We are now able to address the question of distribution of rightmost digits of Mersenne numbers. 
Given a string $\sigma$ of $s$ digits to base $q$, 
\be\label{eq:sigma}
\sigma = (a_{s-1}, \ldots, a_0) \in \{0, \ldots, q-1\}^s
\ee
we denote by $A_r(X,\sigma)$ the number of primes $p \le X$ such that $M_p$ written in base $q$
has $\sigma$ as the string on $s$ consecutive digits on positions $r, \ldots, r-s+1$, counting from the 
right to the left, where the numbering starts with  zero.

We recall that by the prime number theorem (in a very crude form) we have 
$\pi(X) = (1+o(1)) X/\log X$ as $X \to \infty$.

\begin{theorem}
\label{thm:MersDig}
For a fixed prime $q\ge 3$, a real $\varepsilon>0$ and a string $\sigma$ of length $s$ of the form~\eqref{eq:sigma},
uniformly over $\varepsilon \log X \le r \le (\log  X)^{3/2-\varepsilon}$ 
and strings $\sigma$ of length $s$ of the form~\eqref{eq:sigma} we have
$$
A_r(X,\sigma) = \(q^{-s} + o(1)\)\pi(X) 
$$
as  $X \to \infty$. 
\end{theorem} 

We remark that the lower bound on $r$ can be relaxed but a condition of this kind is necessary. 
For example, if $2$ is not a primitive root modulo $q$ the distribution of digits on the rightmost positions 
cannot be uniform. 

\section{Preliminaries}\label{sec:pre}
\subsection{Notation}

Throughout, $\N$ is the set of positive integers.
The letters $k$, $m$ and $n$ (with or without subscripts) are always used to denote positive integers;
the letter~$q$ 
(with or without subscripts) is always used to denote a prime.

Given a prime $q$, let $\padic$ denote the \emph{standard $q$-adic valuation}.
In particular, for every $n\in \Z \setminus \{0\}$ one has $\padic(n)=k$, where $k$ is
the largest nonnegative integer for which $q^k\mid n$.

Given a sequence of complex weights  
$$
\bgamma = (\gamma_h)_{h\in \cH},  
$$
supported on a finite set $\cH$ and $\varsigma\ge 1$   we define norms of $\bgamma$ in the usual way 
$$
\| \bgamma \|_{\infty}=\max_{h\in \cH} |\gamma_h| \quad \text{and} \quad \| \bgamma \|_{\varsigma}= \(\sum_{h\in \cH}  |\gamma_h|^{\varsigma} \)^{\frac{1}{\varsigma}}.
$$

For given functions $F$ and $G$, the notations $F\ll G$, $G\gg F$ and
$F=O(G)$ are all equivalent to the statement that the inequality
$|F|\le c|G|$ holds with some constant $c>0$. Throughout the paper, any implied constants in symbols $O$, $\ll$
and $\gg$ may depend on the parameters $q$, $A$ and are \emph{absolute} unless specified otherwise.

 We write $F\asymp G$
to indicate that $F\ll G$ and $G\ll F$  both hold.

Finally we use $\#\cS$ to denote the cardinality of a finite set $\cS$.

\subsection{Sums over primes}
\label{sec:primes}

It is convenient to use a form of the {\it Vaughan identity\/} given by~\cite[Chapter~24, Equation~(6)]{Dav}.

\begin{lemma}\label{lem:Vaughan's identity}
For any complex-valued function $f(n)$ with 
$|f(n)| \le 1$ and any real numbers $1 < U,V\leq X$ with $UV\leq X$, we have
$$
\sum\limits_{n\leq X}\Lambda(n)f(n)\ll U+\Sigma_1 \log X+ \Sigma_2^{1/2}  X^{1/2} (\log X)^3
$$
where
\begin{align*}
\Sigma_1   & = \sum_{t \leq UV}\max_{w \le X/t} \left|\sum_{w \le  m \leq X/t} f(mt)\right|,\\
\Sigma_2 &=  \max_{U \le w \le X/V} \max_{V \le j \le X/w} \sum_{V < m \le X/w} 
 \left|\sum_{\substack{w < n \le 2w\\ n \leq X/m\\ n \le X/j}} f(jn) \overline f(mn)\right|.
\end{align*} 
\end{lemma}

\subsection{Multiplicative order of integers}
\label{sec:tech-lemmas}

Fix a prime $q\ge 3$ and an integer $g\neq\pm 1$ with $\gcd(g,q) =1$. 
For every $n\in\N$, let $\tau_n=\ord_{q^n} g$ denote the order of $g$ modulo $q^n$.
We write
\be\label{eq:gtn}
g^{\tau_n}=1+h_nq^{n+\mathfrak g_n}\qquad(n\in\N)
\ee
with some uniquely determined integers $h_n$ and $\mathfrak g_n\ge 0$ such that
$\gcd(h_n,q)=1$. 
We also put
\be\label{eq:def tau G}
\tau=\tau_1\mand\gdx=\mathfrak g_1+1=\padic(g^\tau-1).
\ee
A simple argument shows
\be\label{eq:infinity-war}
\mathfrak g_n=\begin{cases}
\gdx-n&\quad\hbox{if $n\le\gdx$},\\
0&\quad\hbox{if $n\ge\gdx$},
\end{cases}
\quad \text{and} \quad 
\tau_n=\begin{cases}
\tau&\quad\hbox{if $n\le\gdx$},\\
q^{n-\gdx}\tau&\quad\hbox{if $n\ge\gdx$}.
\end{cases}
\ee

The following two statements are easy consequences of~\eqref{eq:infinity-war}.

\begin{lemma}
\label{lem:frog1}
For $r\ge s\ge\gdx$ we have
$$
g^{n_1\tau_s}\equiv g^{n_2\tau_s}\bmod{q^r}\quad\Longleftrightarrow\quad
q^{r-s}\mid (n_1-n_2).
$$
\end{lemma}

\begin{lemma}
\label{lem:frog2}
For $m\in\N$ and nonnegative integers $x$ and $y$ with $x\ne y$, either $q\nmid g^{mx}-g^{my}$ or
$$
\padic(g^{mx}-g^{my})=\padic(x-y)+\padic(m)+\gdx.
$$
\end{lemma}

\begin{proof} Put $\tau_0=1$. For any integer $n\ge 0$ we have
$q^n\mid g^{mx}-g^{my}$ if and only if
$mx\equiv my\bmod{\tau_n}$.  Consequently,
$$
\padic(g^{mx}-g^{my})=\max\{n\ge 0:~\tau_n\mid m(x-y)\},
$$
and the result follows from~\eqref{eq:infinity-war} as $\gcd(\tau,q)=1$.
\end{proof}

\subsection{Explicit form of the Vinogradov mean value theorem}

Let $N_{r,k}(P)$ be the number of integral solutions to the system of equations
$$
n_1^j+\cdots + n_r^j=m_1^j+\cdots+m_r^j\qquad
(1\le j\le k,~1\le n_{\ell},m_{\ell}\le P).
$$
Our application of Lemma~\ref{lem:Kor} below requires a precise form of the 
Vinogradov mean value theorem.
For this purpose, we use a fully explicit version
due to Ford~\cite[Theorem~3]{Ford},
which is presented here in a weakened and simplified form. 

\begin{lemma}
\label{lem:Ford}
For any integer $k\ge 129$ there is an integer
$r\in[2k^2,4k^2]$ such that for $P >0$
$$
N_{r,k}(P)\le k^{3k^3}P^{2r-k(k+1)/2+k^2/1000}. 
$$  
\end{lemma} 

We note that the condition $r \ge 2k^2$ is not explicit in~\cite[Theorem~3]{Ford}
however we can always impose this in view of the well-know (and essentially 
trivial) monotonicity property
$$
N_{r+1,k}(P)P^{-2(r+1)} \le N_{r,k}(P) P^{-2r}.
$$ 

We  also observe that the recent striking advances
in the Vinogradov mean value theorem due to 
Bourgain, Demeter and Guth~\cite{BDG} and Wooley~\cite{Wool}  are not suitable
for our purposes here as they contain implicit constants
that depend on $r$ and~$k$, whereas in our approach $r$ and $k$
grow together with $P$. On the other hand, a result of Steiner~\cite{Ste} 
may perhaps be used to improve numerical constants in our estimates in some ranges
of parameters. 

\subsection{Double exponential sums with polynomials}

Our main tool to bound the exponential sum $S_m(a,X)$  is the following variation of a
result of Korobov~\cite[Lemma~3]{Kor}; examining the proof of~\cite[Lemma~3]{Kor} one
can easily see that one can add complex weights   
$\alpha(x)$ and  $\beta(y)$ without any 
changes in the proof.  

It is convenient to denote 
$$
\e(t)=\exp(2\pi i t)\qquad(t\in\R).
$$

\begin{lemma}\label{lem:Kor}
Let $\xi_j\in\R$ for $j=1,\ldots,k$, and suppose that each $\xi_j$
has a rational approximation such that
$$
\left|\xi_j-\frac{b_j}{q_j}\right|\le\frac{1}{q_j^2}\qquad\text{with}\quad
b_j\in\Z,\quad q_j\in\NN,\quad\text{and}\quad(b_j,q_j)=1.
$$
Then, for any natural number $r$ and sequences of complex numbers $\alpha(x),\beta(y)$ satisfying 
$$|\alpha(x)|,|\beta(y)|\le 1,$$
 the sum
$$
S=\sum_{x,y=1}^P \alpha(x)\beta(y)\e\(\xi_1xy+\cdots+\xi_k x^ky^k\),
$$
admits the upper bound
\begin{align*}
|S|^{2r^2}\le \(64r^2\log(3Q)\)^{k/2} &
P^{4r^2-2r} N_{r,k}(P)\\
& \prod_{j=1}^k\min\left\{P^j,P^jq_j^{-1/2}+q_j^{1/2}\right\},
\end{align*}
where 
$$
Q=\max\limits_{1\le j\le k}\{q_j\}.
$$
\end{lemma}

The following result follows from the standard completing technique, see~\cite[Section~12.2]{IwKow}.

\begin{lemma}\label{lem:sum-convert}
For an arbitrary function $f:\R\to\R$,
an interval $\cI$ of length $N$, and integers $U,V$ satisfying 
$$UV\le \frac{N}{2},$$
there exists some $\alpha \in \R$ such that
$$
\sum_{x\in\cI}\e(f(x))  
\ll \frac{\log{N}}{UV}\sum_{x\in\cJ}\sum_{u\le U}
\left|\sum_{\substack{v \le V}}\e(f(x+uv)+\alpha v)\right|, 
$$
and $\cJ$ is some interval of length $2N$. 
\end{lemma}

\begin{proof} It is enough to write 
$$
\sum_{x\in\cI}\e(f(x))  = \sum_{\substack{x \in \cJ\\ x+uv \in \cI}} \e(f(x+uv))  
$$
and the use the completing  technique from~\cite[Section~12.2]{IwKow} to encode the 
condition $x+uv \in \cI$ into linear exponential sums and the use~\cite[Bound~(8.6)]{IwKow}. \end{proof} 
 
\subsection{Bilinear forms with exponential functions}

Fix a prime $q$ and integer $g\neq\pm 1$ with $\gcd(q,g)=1$. 
We denote
by $\tau_n$
the order of $g$ modulo $q^n$, and recall how  $G$ is  defined in~\eqref{eq:def tau G}.

The following proposition is the main ingredient for Theorem~\ref{thm:main}. It 
uses some ideas of Korobov~\cite[Theorem~4]{Kor}.

\begin{proposition}\label{lem:bilinear}
Let $\gamma\in\N$ with $\gamma>16G$.
Given integers $K,L\ge 0$ and $M,N\ge 1$ with
$$
M\le q^{2\gamma/65},
$$
two sequences of complex weights
$$
\balpha = (\alpha_m)_{m=K+1}^{K+M} \mand 
\bbeta = (\beta_n)_{n=L+1}^{L+N}
$$ 
and an integer $z$ not divisible by $q$, 
for the sum
$$
S =  \sum_{m=K+1}^{K+M}  \sum_{n=L+1}^{L+N}  
\alpha_m\beta_n\e_{q^\gamma}(zg^{mn}), 
$$
we have
\begin{align*}
S\ll 
  \|\balpha\|_{2}\|\bbeta\|_\infty & \( M^{1/2-10^{-10}\rho^2}N   \log{M} + M^{1/2}N^{1/2}\) \\ 
  & \qquad \qquad \qquad \qquad\qquad\qquad +\|\balpha\|_{\infty}\|\bbeta\|_{\infty}Nq^{8G},
\end{align*}
where
$$
\rho=\frac{\log M}{\log q^\gamma}.
$$
\end{proposition}

\begin{proof} To simplify the notation, we write
$$
\cM=\{K+1,\ldots,K+M\}\mand\cN=\{L+1,\ldots,L+N\}.
$$
First note we may assume 
\begin{equation}
\label{eq:MLB}
M\ge (\log{q^{\gamma}})^{32},
\end{equation}
as otherwise 
$$
M^{\rho^2}\ll 1,
$$
and hence for the first term in the bound for $S$
$$
\|\balpha\|_{2}\|\bbeta\|_\infty   M^{1/2-10^{-10}\rho^2}N\log{M}\gg \|\balpha\|_{2}\|\bbeta\|_\infty M^{1/2}N\log{M},
$$
which is worse than trivial. If
$$
M\le q^{8G},
$$
then we have
$$
S\le \sum_{m=K+1}^{K+M}  \sum_{n=L+1}^{L+N}  
|\alpha_m||\beta_n| \le \|\balpha\|_{\infty}\|\bbeta\|_{\infty}Nq^{8G}.
$$
Hence we may assume
$$
M\ge q^{8G}.
$$
By the Cauchy--Schwarz inequality
\be\label{eq:afterCS}
\begin{split}
|S|^2 & \le \|\balpha\|_2^2 \sum_{m\in\cM}\left|\sum_{n\in\cN}\beta_n\e_{q^\gamma}(zg^{mn})\right|^2\\
& \le  \|\balpha\|_2^2 \sum_{n_1,n_2\in\cN} |\beta_{n_1}| |\beta_{n_2}||S(n_1,n_2)|\\
& \le   \|\balpha\|_2^2  \|\bbeta\|_\infty^2 \sum_{n_1,n_2\in\cN} |S(n_1,n_2)|,
\end{split} 
\ee
where 
$$
S(n_1,n_2)=\sum_{m\in\cM} \e_{q^\gamma}(z(g^{n_1m}-g^{n_2m})).
$$

Recall we are assuming
\be\label{eq:case1}
\rho\le\frac{2}{65}.
\ee
Define $s$ by 
\be
\label{eq:bilinearsdef}
s=\left\lfloor \frac{\rho \gamma}{8} \right\rfloor
=\left\lfloor\frac{1}{8}\frac{\log M}{\log q}\right\rfloor \ge G,
\ee
so that from~\eqref{eq:infinity-war}, we have
\be\label{eq:tau_is_small}
\tau_s \le q^{s}\le M^{1/8},
\ee
and
\be\label{eq:p^s}
q^s>\frac{M^{1/8}}{q}\gg M^{1/8},
\ee
with implied constant depending on $q$. To establish the desired result we bound $S(n_1,n_2)$ in different ways
as the pair $(n_1,n_2)$ varies over $\cN\times\cN$.

We denote
\begin{align*}
\cA_1&=\{(n_1,n_2)\in\cN\times\cN:\padic(n_1)>s\text{~or~}\padic(n_2)>s\},\\
\cA_2&=\{(n_1,n_2)\in\cN\times\cN:g^{n_1\tau_s}\equiv g^{n_2\tau_s}\bmod{q^{2s}}\},\\
\cA_3&=(\cN\times\cN)\setminus(\cA_1\cup\cA_2).
\end{align*}
Clearly, 
$$
\# \cA_1\le 2N^2/q^s,
$$ and Lemma~\ref{lem:frog1} implies
that 
$$
\# \cA_2\le N^2/q^{s}+N.
$$ 
Thus using the trivial bound
$|S(n_1,n_2)|\le M$ along with~\eqref{eq:p^s} we get that
\be
\label{eq:mathcomp}
\begin{split}
\sum_{j=1,2}\sum_{(n_1,n_2)\in\cA_j}|S(n_1,n_2)|& \ll
  \(\frac{MN^2}{q^s}+MN\)\\
& \ll   \(N^2 M^{7/8}+MN\).
\end{split} 
\ee

For the final set $\cA_3$, we need a nontrivial bound on $S(n_1,n_2)$.
Let $(n_1,n_2)\in\cA_3$ be fixed.  Since $|S(n_1,n_2)|=|S(n_2,n_1)|$,
without loss of generality we can assume
\be\label{eq:ordn12}
\padic(n_1)=a,\qquad \padic(n_2)=b,\qquad a\le b\le s.
\ee
With $a$ and $b$ fixed for the moment, it is convenient to define
\be\label{eq:kPdefs}
k=\fl{\frac\gamma{s+a}}\mand P=q^{s+a}.
\ee
Using the definition of $s$ along with~\eqref{eq:case1} and~\eqref{eq:tau_is_small}
we see that
\be\label{eq:k_is_large_2}
 k\ge 129 \mand P\le q^{2s} \le M^{1/4}.
\ee
Now put $\lambda=g^{n_1}$ and $\mu=g^{n_2}$, so that
$$
S(n_1,n_2)=\sum_{m\in\cM}\e_{q^\gamma}(z(\lambda^m-\mu^m)).
$$
Using~\eqref{eq:gtn},~\eqref{eq:bilinearsdef} and~\eqref{eq:ordn12} it is easy to see that the relations
\be\label{eq:lmt}
\lambda^{\tau_s}=1+uq^{s+a}\mand\mu^{\tau_s}=1+vq^{s+b}
\ee
hold with some integers $u,v$ coprime to $q$.
Partitioning the summation over $m$ into distinct residue
classes modulo $\tau_s$ leads to the estimate
\be\label{eq:S0}
S(n_1,n_2)=S_0(n_1,n_2)+O(\tau_s)=S_0(n_1,n_2)+O(M^{1/8})
\ee
by~\eqref{eq:tau_is_small}, where
$$
S_0(n_1,n_2)=\sum_{x=1}^{\tau_s}
\sum_{y\in\cY}\e_{q^\gamma}(z(\lambda^{x+\tau_sy}-\mu^{x+\tau_sy})),
$$
and
$$
\cY=\big(K/\tau_s,(K+M)/\tau_s\big]\cap\Z.
$$
By~\eqref{eq:lmt} we have 
\begin{align*}
\lambda^{x+\tau_sy}-\mu^{x+\tau_sy}&=\lambda^x(1+uq^{s+a})^y-\mu^x(1+vq^{s+b})^y \\ 
&=\lambda^x\sum_{i=0}^y\binom{y}{i}u^{i}q^{(s+a)i}-\mu^x
\sum_{i=0}^y\binom{y}{i}v^{i}q^{(s+b)i}\\
&\equiv \lambda^x-\mu^x+
\sum_{i=1}^\kdx q^{(s+a)i}(\lambda^xu^i-\mu^xv^iq^{\Delta i})\binom{y}{i}
\bmod{q^\gamma},
\end{align*}
where we have put $\Delta=b-a$ (note that~\eqref{eq:kPdefs} is used
in the last step); therefore,
$$
|S_0(n_1,n_2)|\le \sum_{x=1}^{\tau_s}\left|
\sum_{y\in\cY}\e_{q^\gamma}
\(\,\sum_{i=1}^\kdx q^{(s+a)i}
(\lambda^xu^i-\mu^xv^iq^{\Delta i})\binom{y}{i}\)\right|.
$$
We apply Lemma~\ref{lem:sum-convert} with the function
$$
f(y)=\sum_{i=1}^\kdx q^{(s+a)i}
(\lambda^xu^i-\mu^xv^iq^{\Delta i})\binom{y}{i},
$$
and parameters
$$
U=V=P, \quad \cI=\cY,
$$
and note by~\eqref{eq:tau_is_small} and~\eqref{eq:k_is_large_2} 
$$
P^2\le M^{1/2}\le M^{7/8}\le \frac{M}{\tau_s}\le \# \cY+1.
$$
It follows that
\begin{equation}\label{eq:sum_split}
\begin{split} 
S_0&(n_1,n_2) \\&\ll \frac{\log{M}}{P^2}\sum_{x=1}^{\tau_s}\sum_{y\in\cZ}\sum_{z_1=1}^{P}
\left|\sum_{z_2=1}^{P} 
 \e(\alpha_x z_2) \e_{q^\gamma}(f(y+z_1z_2))\right| \\
& \ll \frac{\log{M}}{P^2}\sum_{x=1}^{\tau_s}\sum_{y\in\cZ}
\left|\sum_{z_1=1}^{P}\sum_{z_2=1}^{P}  \e(\alpha_x z_2)\beta_{x,y}(z_1)
\e_{q^\gamma}(f(y+z_1z_2)) \right|, 
\end{split} 
\end{equation}
where $\cZ$ is an interval of length $O(M/\tau_s)$ 
and $\alpha_x$ 
may depend on the variable $x$
and $\beta_{x,y}$ may depend on the variables $x$ and $y$ and satisfies 
$$
|\beta_{x,y}(z_1)|=1.
$$
With the intention of applying Lemmas~\ref{lem:Ford} and~\ref{lem:Kor}
to the right side of~\eqref{eq:sum_split}, we
fix $y$ for the moment and write
$$
k!  f(y+Z)=\sum_{j=0}^k a_jZ^j\qquad(a_j\in\Z)
$$
and for each $i=1,\ldots, k$,
\be\label{eq:expand2}
k!  q^{(s+a)i}(\lambda^xu^i-\mu^xv^iq^{\Delta i})\binom{y+Z}{i}
=\sum_{j=1}^i a_{i,j}Z^j
\ee
with some $a_{i,j}\in\Z$.
Clearly, 
$$
a_j=\sum_{i=j}^k a_{i,j},
$$
 and thus
\be\label{eq:padics-aij}
\padic(a_j)\ge\min\{\padic(a_{i,j}):~i=j,\ldots,k\}.
\ee
Moreover, equality holds in~\eqref{eq:padics-aij} whenever
\be\label{eq:padic-cond}
\padic(a_{j,j})<\padic(a_{i,j})\qquad(i>j).
\ee
Denote
$$
\overline\nu=\min\left\{\padic(\lambda^xu^j-\mu^xv^jq^{\Delta j})
:~j=1,\ldots,k\right\},
$$
and let $\overline{j}$ be an index for which
\be\label{eq:optimal-j}
\padic(\lambda^xu^{\overline{j}}-\mu^xv^{\overline{j}}q^{\Delta {\overline{j}}})=\overline\nu.
\ee
From~\eqref{eq:expand2} it is clear that
$$
a_{{\overline{j}},{\overline{j}}}=\frac{k!}{{\overline{j}}!} \, q^{(s+a)\overline{j}}(\lambda^xu^{\overline{j}}-\mu^xv^{\overline{j}}q^{\Delta {\overline{j}}}),
$$
and therefore
\be\label{eq:padic-ajj}
\padic(a_{{\overline{j}},{\overline{j}}})=\padic(k!)-\padic({\overline{j}}!)+(s+a){\overline{j}}+\overline\nu.
\ee
On the other hand,~\eqref{eq:expand2} implies
\be\label{eq:padic-aij}
\padic(a_{i,{\overline{j}}})\ge\padic(k!)-\padic(i!)+(s+a)i+\overline\nu
\qquad(i>{\overline{j}}).
\ee
Before we proceed, we note that the 
estimate ${\overline{j}}<i\le k<q^{s+a}$ holds since by~\eqref{eq:MLB},~\eqref{eq:bilinearsdef} and~\eqref{eq:kPdefs} we have
\begin{equation}
\label{eq:kbounds}
k\le \frac{\gamma}{s}\le \frac{2\gamma}{s+1} < \frac{16}{\rho}=\frac{16\log{q^{\gamma}}}{\log{M}}\le  M^{1/32}< q^{s}\le q^{s+a}.
\end{equation}
This implies the inequality
$$
(s+a)(i- {\overline{j}})>\padic(i(i-1)\cdots  (\overline{j}+1))=\padic(i!)-\padic({\overline{j}}!),
$$
which together with~\eqref{eq:padic-ajj} and~\eqref{eq:padic-aij}
verifies the condition~\eqref{eq:padic-cond} for any ${\overline{j}}$ satisfying~\eqref{eq:optimal-j}. 
Hence,~\eqref{eq:padics-aij} holds with equality, and thus we have
\be\label{eq:optimal-j-two}
\padic(a_{\overline{j}})=\padic(k!)-\padic({\overline{j}}!)+(s+a){\overline{j}}+\overline\nu
\ee
for any ${\overline{j}}$ satisfying~\eqref{eq:optimal-j}.

If $\Delta>0$, then clearly
$$
\padic(\lambda^xu^j-\mu^xv^jq^{\Delta j})=0\qquad(j\ge 1).
$$
For $\Delta=0$ (that is, $a=b$) we claim that for any two consecutive
indices $j$ and $j+1$,
\be\label{eq:parity}
\padic(\lambda^xu^j-\mu^xv^j)= \overline\nu \qquad \text{or} \qquad \padic(\lambda^xu^{j+1}-\mu^xv^{j+1})=\overline\nu.
\ee
To prove the claim, suppose on the contrary that
$$
\lambda^xu^j \equiv \mu^xv^j  \mod q^{\overline\nu+1} \mand \lambda^xu^{j+1}\equiv \mu^xv^{j+1}
\mod q^{\overline\nu+1}
$$
for some $j$. Then, dividing the second conruence by the fisrt one, we get
$u\equiv v \mod q^{\overline\nu+1}$ and thus
$$
\lambda^xu^j \equiv \mu^xv^j  \mod q^{\overline\nu+1} \quad \text{for all } j,
$$
which contradicts the definition of~$\overline\nu$.

Now let
$$
\cJ=\left\{(k+1)/2\le j\le k:~
\padic(\lambda^xu^j-\mu^xv^jq^{\Delta})=\overline\nu\right\}.
$$
In view of~\eqref{eq:parity} this implies that $\# \cJ\ge\fl{k/4}$. 
Since $\lambda^x-\mu^x=g^{n_1x}-g^{n_2x}$
and $n_1\ne n_2$ (in fact, $\padic(n_1-n_2)<s$ by Lemma~\ref{lem:frog1}
since $(n_1,n_2)\not\in\cA_2$), by Lemma~\ref{lem:frog2} and inequalities \eqref{eq:bilinearsdef} and \eqref{eq:tau_is_small} we have
$$
\padic(\lambda^x-\mu^x)=0\qquad\text{or}\qquad
\padic(\lambda^x-\mu^x)=\padic(n_1-n_2)+\padic(x)+\gdx\le 3s;
$$
this implies that $\overline\nu\le 3s$.
Thus, for every $j\in\cJ$ we have by~\eqref{eq:optimal-j-two}:
$$
(s+a)j\le\padic(a_j)\le\padic(k!)+(s+a)j+3s,
$$
and so (recalling that $P=q^{s+a}$) we can write
\be \label{eq:qj}
\frac{a_j}{k!q^\gamma}=\frac{b_j}{q_j}
\ee
with
\be \label{eq:fraction}
\gcd(b_j,q_j)=1   \mand 
P^{-j}q^{\gamma-3s}\le q_j\le k!P^{-j}q^\gamma.
\ee 

We also define $q_j $  by~\eqref{eq:qj} for $j \not \in \cJ$. 

We are now in a position to apply Lemmas~\ref{lem:Ford} and~\ref{lem:Kor}
in order to bound the double sum over $z_1$ and $z_2$ in~\eqref{eq:sum_split}.
Writing
\begin{align*}
T&=\sum_{z_1,z_2=1}^{P}\beta_{x,y}(z_1)\alpha_x(z_2)\e_{q^\gamma}(f(y+z_1z_2))
 \\ & =\sum_{z_1,z_2=1}^{P}\beta_{x,y}(z_1)\alpha_x(z_2)\e\bigg(\sum_{j=1}^k \frac{b_j}{q_j}(z_1z_2)^j\bigg),
\end{align*}
Lemma~\ref{lem:Kor} shows that for any natural number $r$, the bound
\begin{align*}
|T|^{2r^2}\le \(64r^2\log(3Q)\)^{k/2}&
P^{4r^2-2r} N_{r,k}(P)\\
& \prod_{j=1}^k
\min\left\{P^j,P^jq_j^{-1/2}+q_j^{1/2}\right\}
\end{align*}
holds with
$Q=\max\limits_{1\le j\le k} q_j$. Note that~\eqref{eq:kbounds} and~\eqref{eq:fraction} imply that 
$$
\log(3Q)\le\log(3k!q^\gamma)\le  \gamma\log (kq) \le \gamma k \log q
$$
since for $129 \le k \le \gamma$ we have $3k! \le k^k \le k^\gamma$. 
Moreover, since $k\ge 129$ (see~\eqref{eq:k_is_large_2})
Lemma~\ref{lem:Ford} shows that we can choose the integer
$r\in [2k^2,4k^2]$ so that
$$
N_{r,k}(P)\le k^{3k^3}P^{2r-k(k+1)/2+k^2/1000}.
$$
Hence we find that
\be\label{eq:sigma2r}
|T|^{2r^2}\le \(1024\gamma k^{5}\log q\)^{k/2}
k^{3k^3+3k}P^{4r^2-k(k+1)/2+k^2/1000}R,
\ee 
where
\begin{align*}
R& =\prod_{j=1}^k\min\left\{P^j,P^jq_j^{-1/2}+q_j^{1/2}\right\}\\
&=
P^{k(k+1)/2}\prod_{j=1}^k\min\left\{1,q_j^{-1/2}+P^{-j}q_j^{1/2}\right\}.
\end{align*}
For any $j\in\cJ$ we have $j\ge (k+1)/2$. Recalling~\eqref{eq:kPdefs},  we have
$$
P^{-j}\le P^jq^{-\gamma};
$$
thus, using~\eqref{eq:fraction} we see that
\begin{align*}
q_j^{-1/2}+P^{-j}q_j^{1/2} & \le P^{j/2}q^{-\gamma/2+3s/2}
+(k!)^{1/2}P^{j/2}q^{-\gamma/2}\\
& \le k^kP^{j/2}q^{-\gamma/2+3s/2}.
\end{align*}
For  $j\not \in\cJ$ we use the trivial bound 
$$
\min\left\{1,q_j^{-1/2}+P^{-j}q_j^{1/2}\right\}\le 1.
$$ 

Therefore, recalling that $\# \cJ\ge\fl{k/4}$,  
and using the bounds
$$
0.24k<\fl{k/4}\le k/4\mand\sum_{j=k-\fl{k/4}+1}^k j/2<0.11k^2
$$
which hold for $k\ge 129$,
we see that
\begin{align*}
R&\le 
P^{k(k+1)/2}\prod_{j\in\cJ}
\(k^kP^{j/2}q^{-\gamma/2+3s/2}\)\\
&\le k^{k^2}P^{k(k+1)/2}\prod_{j=k-\fl{k/4}+1}^k
\(P^{j/2}q^{-\gamma/2+3s/2}\)\\
&\le k^{k^2}P^{k(k+1)/2+0.11k^2}q^{-0.12\gamma k+3sk/8}.
\end{align*}
Combining this bound with~\eqref{eq:sigma2r} we deduce that
\be
\label{eq:TABC}
|T|\le (ABC)^{1/2r^2}P^2,
\ee
where
$$
A=2^{5k}k^{3k^3+k^2+11k/2},\quad
B=\(\gamma\log q\)^{k/2},\quad
C=P^{0.111k^2}q^{-0.12\gamma k+3sk/8}.
$$
Since $r\geq 2k^2$  
it is clear that 
\be
\label{eq:Abound}
A^{1/2r^2}\ll 1.
\ee
Next, since $k\asymp\gamma/s\asymp\rho^{-1}$ we have
$$
\gamma\log q=\rho^{-1}\log M\ll k\log M,
$$
hence
\be
\label{eq:Bbound}
B^{1/2r^2}\ll(k \log M)^{1/8k^4}\ll \log{M}.
\ee 
Recalling~\eqref{eq:bilinearsdef} and~\eqref{eq:kPdefs}, we have
$$\frac{\gamma}{s+a}-1<k\le\frac{\gamma}{s+a} \quad  \text{and} \quad  \frac{\gamma}{s}\ge \frac{8}{\rho},$$ 
and using~\eqref{eq:bilinearsdef}, we get that
\begin{align*}
\frac{\log C}{\log q}&=0.111(s+a)k^2-0.12\gamma k+3sk/8\\
&\le -\frac{0.009\gamma^2}{s+a}+0.12\gamma+\frac{3s\gamma}{8(s+a)}\\
&\le -\frac{0.009\gamma^2}{s}+0.12\gamma+\frac{3\gamma}{8}\\
& \le -\frac{0.036 \gamma}{\rho}+0.495 \gamma \le -\frac{0.02 \gamma}{\rho},
\end{align*}
where we have used the inequality $\rho\le 2/65$ in the last step;
thus, 
\be 
\label{eq:C prelim}
C\le M^{-0.02/\rho^2}.
\ee  
Since
$$ r \le 4k^2,$$
we have 
\begin{equation}\label{eq:kin0}
\frac{0.02}{2r^2\rho^2}\ge \frac{1}{1600 \rho^2 k^4},
\end{equation}
and from~\eqref{eq:bilinearsdef} and~\eqref{eq:kPdefs}
\begin{equation}
\label{eq:kin1}
k\le \frac{\gamma}{s+a}\le \frac{\gamma}{\rho\gamma/8-1}.
\end{equation} 

Since 
$$
\rho\gamma=\frac{\log{M}}{\log{q}},
$$
and we allow the implied constant in the statement of Proposition~\ref{lem:bilinear} to depend on $q$, 
we may assume that  $M \ge q^{16}$ and thus
$$
\rho \gamma \ge 16,
$$
which combined with~\eqref{eq:kin1} implies 
$$
k\le \frac{16}{\rho},
$$
and hence by \eqref{eq:kin0}
$$
\frac{0.02}{2r^2\rho^2}\ge  \frac{1}{1600 k^4 \rho^2} \ge  \frac{1}{25\cdot 2^{22}} \rho^2 \ge  10^{-9}\rho^2.
$$

Substituted in~\eqref{eq:C prelim}, this gives  
$$C^{1/2r^2}\le M^{-10^{-9}\rho^2}.$$ 
Combining the above with~\eqref{eq:TABC},~\eqref{eq:Abound} and~\eqref{eq:Bbound} we get 
$$
T\ll P^2M^{-10^{-9}\rho^2} \log{M}.
$$

Inserting the previous bound into~\eqref{eq:sum_split}
and using~\eqref{eq:k_is_large_2} we have
$$
S_0(n_1,n_2) \ll \tau_s \# \cY M^{-10^{-9}\rho^2} (\log{M})^2
\ll M^{1-10^{-9}\rho^2} (\log{M})^2,
$$  
since 
$$\tau_s \# \cY \ll M.$$
Combining with~\eqref{eq:S0} implies that
\be\label{eq:final}
S(n_1,n_2)\ll M^{1-10^{-9}\rho^2}  (\log{M})^2.
\ee
Now~\eqref{eq:afterCS}, \eqref{eq:mathcomp} and~\eqref{eq:final} together yield the bound
\begin{align*}
S& \ll   \| \balpha\|_2  \|\bbeta\|_\infty NM^{1/2-10^{-10}\rho^2} \log{M}\\
& \qquad  \qquad  \qquad + \| \balpha\|_2 \|\bbeta\|_\infty \(M^{7/16} N +M^{1/2}N^{1/2}\),
\end{align*} 
since $ M^{7/16}N$ never dominates the term $NM^{1/2-10^{-10}\rho^2} \log{M}$,  we obtain the  desired result.
\end{proof}
We may remove the condition $M\le q^{2\gamma/65}$ by partitioning the summation 
over $M$ into short intervals and this is done for applications to Theorem~\ref{thm:MersDig} where we need to considered both large and short ranges of the parameter $M$.

\begin{cor}\label{cor:bilinear}
Let $\gamma\in\N$ with $\gamma>16G$ and let $A>0$ be arbitrary.
Given integers $K,L\ge 0$ and $M,N\ge 1$ with
\begin{equation}
\label{eq:Mubdef}
M\le q^{A\gamma},
\end{equation}
two sequences of complex weights
$$
\balpha = (\alpha_m)_{m=K+1}^{K+M} \mand 
\bbeta = (\beta_n)_{n=L+1}^{L+N}
$$ 
and an integer $z$ not divisible by $q$, 
for the sum
$$
S =  \sum_{m=K+1}^{K+M}  \sum_{n=L+1}^{L+N}  
\alpha_m\beta_n\e_{q^\gamma}(zg^{mn}), 
$$
we have
\begin{align*}
S\ll 
  \|\balpha\|_{2}\|\bbeta\|_\infty & \( M^{1/2-c\rho^2}N   \log{M} + M^{1/2}N^{1/2}\) \\ 
  &  \qquad \qquad\qquad\qquad +\left(1+\frac{M}{q^{2\gamma/65}}\right)\|\balpha\|_{\infty}\|\bbeta\|_{\infty}Nq^{8G},
\end{align*}
where $$
\rho=\frac{\log M}{\log q^\gamma}
$$
and $c>0$ is a constant depending on $A$.
\end{cor} 

\begin{proof}
By Proposition~\ref{lem:bilinear} we may assume $M\ge q^{2\gamma/65}$, 
and by modifying the coefficients $\balpha$  
(appending them with at most  $\lfloor q^{2\gamma/65}\rfloor$ zeros) 
we may assume 
\begin{equation}
\label{eq:MJb}
M=JM_0, \quad \text{with} \quad  M_0=\lfloor q^{2\gamma/65}\rfloor
\end{equation} 
for some integer $J\geq 1$.
Subdividing $S$ into $J$ sums
$$
S_j=\sum_{m=K+1+M_0j}^{K+M_0(j+1)}  \sum_{n=L+1}^{L+N}  
\alpha_m\beta_n\e_{q^\gamma}(zg^{mn}),
$$
by the the Cauchy--Schwarz inequality and   Proposition~\ref{lem:bilinear} (applied for each $0\le j \le J-1$),
denoting 
where 
\begin{equation}\label{eq:rho_0}
\rho_0=\frac{\log M_0}{\log q^{\gamma}}.
\end{equation}
we obtain 
\begin{align*}
|S|^2& \le J\sum_{j=0}^{J-1}|S_j|^2\\
& \ll  J
 \|\bbeta\|^2_\infty \sum_{j=0}^{J-1} \sum_{m=K+1+M_0j}^{K+M_0(j+1)}|\alpha_m|^{2} \\
 & \qquad\qquad  \qquad\qquad  \( M_0^{1-2\cdot 10^{-10}\rho_0^2}N^2   \log^2{M} +
  q^{2\gamma/65}N\) \\ 
 & \qquad \qquad\qquad\qquad \qquad\qquad  \qquad\qquad  +  J^2\|\balpha\|^2_{\infty}\|\bbeta\|^2_{\infty}N^2q^{16G}\\
 &\ll 
  \|\balpha\|^2_{2}\|\bbeta\|^2_\infty   \( Jq^{2\gamma(1-2\cdot 10^{-10}\rho_0^2)/65}N^2   \log^2{M} + Jq^{2\gamma/65}N\) \\ 
  & \qquad \qquad\qquad\qquad \qquad\qquad  \qquad\qquad +J^2\|\balpha\|^2_{\infty}\|\bbeta\|^2_{\infty}N^2q^{16G}\\
 & \ll 
  \|\balpha\|^2_{2}\|\bbeta\|^2_\infty   \( Mq^{-2\cdot 10^{-10}\gamma \rho_0^2/65}N^2   \log^2{M} + MN\) \\ 
  & \qquad \qquad\qquad\qquad \qquad\qquad \qquad\qquad +\|\balpha\|^2_{\infty}\|\bbeta\|^2_{\infty}N^2\frac{q^{16G}M^2}{q^{4\gamma/65}}.
\end{align*}
By~\eqref{eq:Mubdef}, \eqref{eq:MJb} and \eqref{eq:rho_0} we have 
$$
q^{10^{10}\gamma \rho_0^2}\ge M^{c \rho^2},
$$
for some constant $c$ depending on $A$. Hence 
\begin{align*}
|S|\ll 
  \|\balpha\|_2\|\bbeta\|_\infty & \( M^{1/2-c\rho^2}N   \log^2{M} + M^{1/2}N^{1/2}\) \\ 
  & \qquad \qquad\qquad\qquad +\|\balpha\|_{\infty}\|\bbeta\|_{\infty}N\frac{q^{8G}M}{q^{2\gamma/65}},
\end{align*}
which completes the proof.
\end{proof}

We now estimate double sums with variables limits of summation for one variable. 

\begin{lemma}
\label{lem:DoubleExpVarLim}
Let $\gamma\in\N$ with $\gamma>16G$ and let $A>0$ be arbitrary.
Given integers   $M,N\ge 1$ and $L\geq 0$ with
$$
M\le q^{A\gamma}, 
$$
two sequences
$$
(K_m)_{m=1}^M \mand (N_m)_{m=1}^M
$$
of nonnegative integers such that $K_m< N_m\le N$ for each $m$,
two sequences of complex weights
$$
\balpha = (\alpha_m)_{m=1}^{M} \mand 
\bbeta = (\beta_n)_{n=1}^{N}
$$ 
with 
$$
 \| \balpha\|_\infty,  \|\bbeta\|_\infty \ll 1
$$
and an integer $z$ not divisible by $q$, 
for the sum
$$
\widetilde S =  \sum_{m= L+1}^{L+M}  \sum_{K_m \le n \le N_m} \alpha_m \beta_n \e_{q^\gamma}(zg^{mn})  
$$
we have
$$
\widetilde S \ll 
  \left(NM^{1-c\rho^2}+N^{1/2}M\right) \log{M}\log{N} +\left(1+\frac{M}{q^{2\gamma/65}}\right)Nq^{8G}\log{N},
$$
where  $$
\rho=\frac{\log M}{\log q^\gamma}
$$
and $c>0$ is a constant depending on $A$.
\end{lemma} 

\begin{proof}
Using  the standard completing technique, see~\cite[Section~12.2]{IwKow}, and~\cite[Bound~(8.6)]{IwKow} it follows that
$$
\widetilde S =  \sum_{- N/2<r\le  N/2} \frac{1}{|r|+1}
 \sum_{m= L+1}^{L+M}  \sum_{n =1}^N \widetilde\alpha_{m,r}  \widetilde \beta_{n,r}
  \e_{q^\gamma}(zg^{mn}),
$$
where 
$$\widetilde\alpha_{m,r} =  \alpha_m \eta_{m,r} \mand  \widetilde \beta_{n,r}=  \beta_n \e_N(rn),
$$
for some complex number $\eta_{m,r}\ll 1.$
Applying Corollary~\ref{cor:bilinear} 
and noting that
$$
\sum_{- N/2<r\le  N/2} \frac{1}{|r|+1} \ll \log N,
$$
we derive 
$$
\widetilde S \ll 
  \left(NM^{1-c\rho^2}+N^{1/2}M\right) \log{M}\log{N}+ \left(1+\frac{M}{q^{2\gamma/65}}\right)Nq^{8G}\log{N},
$$
which completes the proof.
\end{proof}

\subsection{Bounds on double   exponential sums over hyperbolic domains}

One of our main technical tool is the following result, which gives a 
  bound on double exponential sums over   certain ``hyperbolic''
regions of summation.

We recall the definition of $G$, given in~\eqref{eq:def tau G}.

\begin{lemma}
\label{lem:DoubleSumHyperb} 
Let $\gamma\in\N$ with $\gamma>16G$ and $A>0$.
Given real numbers $X,Y,Z\ge 1$ with
$$
Z  <Y\le q^{A\gamma}, 
$$
and a sequence $\bbeta = (\beta_n)_{n\le X/Z}$
 of complex numbers with 
 $$
 \|\bbeta\|_\infty \le 1
 $$
any sequences
$$
(K_m)_{m=1}^M \mand (N_m)_{m=1}^M 
$$ 
of nonnegative integers such that $K_m<N_m \leq X/m$ for each $m$,
and any integer $z$ coprime to $q$, we have
\begin{align*}
\sum_{Z< m \le Y}& \left| \sum_{K_m \leq n \leq N_m}
\beta_n \e_{q^\gamma}(zg^{mn})\right| \\
&  \ll  \left(XZ^{-c\zeta^2} + (YX)^{1/2}\right) (\log X)^2  +  \left(\frac{1}{Z}+\frac{1}{q^{2\gamma/65}}\right)Xq^{8G}\log{X},
\end{align*}  
where 
\begin{align}
\label{eq:zetadef}
\zeta=\frac{\log Z}{\log q^\gamma}
\end{align}
and $c> 0$ is a constant  depending only $A$. 
\end{lemma}

\begin{proof}
Clearly there are complex numbers $\alpha_m$  such that 
$|\alpha_m| = 1$ for $Z< m \le Y$ and 
 $\alpha_m=0$ otherwise, such that 
 $$
 \sum_{Z< m \le Y} \left| \sum_{K_m \leq n \leq N_m}
\beta_n \e_{q^\gamma}(zg^{mn})\right|
= \sum_{Z< m \le Y} \alpha_m  \sum_{K_m \leq n \leq N_m}
\beta_n \e_{q^\gamma}(zg^{mn}) .
$$ 

Furthermore 
\begin{align*}
\sum_{Z< m \le Y} & \alpha_m  \sum_{K_m \leq n \leq N_m}
\beta_n \e_{q^\gamma}(zg^{mn}) 
\\ & \qquad =
\sum_{\log Z-1\le j\le\log Y}\sum_{e^j < m\le e^{j+1}}
 \sum_{K_m \leq n \leq N_m} \alpha_m 
\beta_n \e_{q^\gamma}(zg^{mn})
\end{align*}
 and we have set $\alpha_m=0$ if $m\le Z$ or $m\geq Y$. 
We observe that for each $j$ within the summation range, we have 
$$
\frac{\log  (e^{j+1}-e^j)}{\log q^\gamma} \ge \frac{\log (Z-1)}{\log q^\gamma} \ge \frac{\zeta}{2},
$$
where $\zeta$ is given by~\eqref{eq:zetadef}. Hence
\begin{align*}
\sum_{e^j < m\le e^{j+1}}&
 \sum_{K_m \leq n \leq N_m} \alpha_m 
\beta_n \e_{q^\gamma}(zg^{mn})  \\
&\quad\ll  \left(\frac{X}{e^j} e^{j(1-c\zeta^2/4)}+e^{j}\left(\frac{X}{e^{j}}\right)^{1/2}  \right)(\log X)^2 \\ & \quad \quad \quad \quad \quad + \left(1+\frac{2^{j}}{q^{2\gamma/65}}\right)\frac{X}{2^j}q^{8G}\log{N}
\end{align*} 
by Lemma~\ref{lem:DoubleExpVarLim} 
and the result follows after renaming $c$, summing the above over $j$ satisfying $\log{Z}-1\le j \le \log{Y}$ and using the estimates 
$$
\sum_{\log Z\le j\le\log Y} e^{-\alpha j}\ll Z^{-\alpha} \mand \sum_{\log Z\le j\le\log Y}e^{j \alpha}\ll Y^{\alpha}.
$$ 
provided $\alpha>0$ is bounded away from $0$. 
\end{proof}

\subsection{Bounds on single exponential sums}
Combining Proposition~\ref{lem:bilinear} with Lemma~\ref{lem:sum-convert} allows us to estimate sums over an interval which has previously been considered by Korobov~\cite[Theorem~4]{Kor}. We present a proof for completeness.

\begin{lemma}
\label{lem:singlesum1}
With notation as in~\eqref{eq:def tau G} and Proposition~\ref{lem:bilinear} suppose $M$  satisfies
$$M\le q^{2\gamma/65}.$$
Then we have 
$$
\sum_{m=K+1}^{K+M}\e_{q^{\gamma}}(zg^m)\ll M^{1-10^{-11}\rho^2}  (\log{M})^2 +M^{10^{-10}\rho^2}q^{8G} \log M,
$$
where 
$$
\rho=\frac{\log{M}}{\log{q^{\gamma}}}.
$$
\end{lemma}

\begin{proof}
Let 
$$S=\sum_{m=K+1}^{K+M}\e_{q^{\gamma}}(zg^m),$$
and apply Lemma~\ref{lem:sum-convert} with 
$$
U=M^{1-10^{-10}\rho^2}, \quad V=M^{10^{-10}\rho^2},
$$
to get 
$$
S\ll \frac{\log{N}}{M}\sum_{m=K+1}^{K+M}\sum_{u\le U}\left|\sum_{v\le V}\e(\alpha v)\e_{q^{\gamma}}(zg^{m}g^{uv})\right|.
$$

Taking a maximum over $m$ in the above, we get 
$$
S\ll \log{M}\sum_{u\le U}\sum_{v\le V}\alpha(u)\beta(v)\e_{q^{\gamma}}(z_0g^{uv}),
$$
for some $\gcd(z_0,p)=1$ and complex numbers $\alpha,\beta$ satisfying 
$$|
\alpha(u)|, |\beta(v)|\le 1.
$$
With 
$$
\rho_0=\frac{\log{U}}{\log{q^{\gamma}}},
$$
we have 
$$
\rho_0= \rho\left(1-10^{-10}\rho^2\right),
$$
hence by Proposition~\ref{lem:bilinear}
\begin{align*}
S &\ll (\log{M})^2\left(V(U^{1-10^{-10}\rho_0^2}+q^{8G})+U^{1/2}V^{1/2}\right)
\\ &\ll (\log{M})^2M\left( M^{-10^{-10}\rho^2(1-10^{-10}\rho^2)^2}+M^{-\frac{1}{2}-\frac{1}{2}10^{-10}\rho^2} \right).
\end{align*}
Note the assumption 
$$M\le q^{2\gamma/65},$$
implies that 
$$\rho\le \frac{2}{65},$$
and hence 
$$
(1-10^{-10}\rho^2)^2\geq \left(1-10^{-10}\left(\frac{2}{65}\right)^2\right)^2\geq \frac{1}{10}
$$
which completes the proof. 
\end{proof}

Partitioning the summation into small intervals as in the proof of Corollary~\ref{cor:bilinear} allows us again to remove the restriction $M\le q^{2\gamma/65}$ in Lemma~\ref{lem:singlesum1}.

\begin{cor}
\label{cor:singlesum1}
With notation as in~\eqref{eq:def tau G} and Proposition~\ref{lem:bilinear} suppose $M$  satisfies
$$M\le q^{A\gamma}.$$
Then we have
$$
\sum_{m=K+1}^{K+M} \e_{q^{\gamma}}(zg^m) \ll M^{1-c\rho^2}  \log{M} +M^{1-c}q^{8G},
$$
where 
$$
\rho=\frac{\log{M}}{\log{q^{\gamma}}}
$$
and $c> 0$ is a constant  depending only $A$. 
\end{cor}  

\begin{proof}
Arguing as in the proof of Corollary~\ref{cor:bilinear}, we may partition the summation over $m$ into intervals of length at most $q^{2\gamma/65}$ and apply Lemma~\ref{lem:singlesum1} to each of these intervals. This produces a bound of the form 
\begin{equation}
\label{eq:b112}
\sum_{m=K+1}^{K+M} \e_{q^{\gamma}}(zg^m) \ll M^{1-c\rho^2}  (\log{M})^2 +M^{1-c}q^{8G} \log M,
\end{equation}
for a constant $c$ depending on $A$. Unless  we have $M^{c\rho^2} \ge  (\log{M})^2$ the estimate~\eqref{eq:b112}  is trivial.   Under this condition we have 
$$
M^{-c\rho^2}   (\log{M})^2 \le \sqrt{M^{-c\rho^2}  (\log{M})^2},
$$
which allows us to replace   $(\log{M})^2$ with $\log{M}$ after changing the constant $c> 0$.
Reducing $c$ if necessary, we can also discard  $\log{M}$ in the second term. 
\end{proof}

\section{Proofs of Main Results} 

\subsection{Proof of Theorem~\ref{thm:main}}

We apply Lemma~\ref{lem:Vaughan's identity} with 
\begin{equation}
\label{eq:mainUVchoice}
U=X^{1/4}, \quad V=X^{1/4},
\end{equation} 
 to get 
\begin{equation}
\label{eq:Vaughan}
S_{q^{\gamma}}(a;X) \ll X^{1/4}+\Sigma_1(\log{X})+\Sigma_2^{1/2}X^{1/2}(\log{X})^3,
\end{equation}
where 
$$
\Sigma_1=\sum_{t\le UV}\max_{w\le X/t}\left|\sum_{w\le m \le X/t}\e_{q^{\gamma}}(ag^{tm})\right|,
$$ 
 and 
$$
\Sigma_2=\max_{U \le w \le X/V} \max_{V \le j \le X/w} \sum_{V < m \le X/w} 
 \left|\sum_{\substack{w < n \le 2w\\ n \leq X/m\\ n \le X/j}} \alpha_n \e_{q^{\gamma}}(ag^{mn})\right|,
$$
for some $|\alpha_n|\le 1$. Considering $\Sigma_1$, for each fixed $t\le UV=X^{1/2}$, define 
$$
G_t=\nu_q(g^{t \ord_q(g^t)}-1)
$$
and
$$
\rho_t=\frac{\log{(X/t)}}{\log{q^{\gamma}}}.
$$
By~\eqref{eq:mainUVchoice} and $t \le UV = X^{1/2}$ we have  
\begin{equation}
\label{eq:rho111}
\rho_t\ge \frac{\rho}{2}.
\end{equation}

We claim, that the following inequality holds
\begin{equation}
  \begin{split}
\label{eq:sigma1nontrivial} 
\max_{w\le X/t} \left|\sum_{w\le m \le X/t}\e_{q^{\gamma}}(ag^{tm})\right| & \\
\ll\left(\frac{X}{t}\right)^{1-c\rho_t^2}  & \log{X}
    + \left(\frac{X}{t}\right)^{1-c}q^{8G_t}.
  \end{split}
\end{equation}  
Indeed, if $\gamma > 16 G_t$ it follows from Corollary~\ref{cor:singlesum1}.

If $\gamma \le  16 G_t$ then 
$$
 \left(\frac{X}{t}\right)^{1-c}q^{8G_t} \ge  \left(\frac{X}{t}\right)^{1-c}q^{\gamma/2} \ge  \left(\frac{X}{t}\right)
$$
so~\eqref{eq:sigma1nontrivial} is trivially true as well
since 
\begin{equation}
\label{eq:sigma1trivial}
 \max_{w\le X/t}\left|\sum_{w\le m \le X/t}\e_{q^{\gamma}}(ag^{tm})\right|\ll \frac{X}{t},
\end{equation}
which proves \eqref{eq:sigma1nontrivial}.

Summing \eqref{eq:sigma1nontrivial} over $t\le UV$ and using~\eqref{eq:sigma1nontrivial}, \eqref{eq:rho111} and~\eqref{eq:sigma1trivial} gives 
\be
\label{eq:sigma1sigma}
\Sigma_1\ll  \sum_{t\le X^{1/2}}\left(\frac{X}{t}\right)^{1-c\rho_t^2} \log X+\widetilde \Sigma_1,
\ee
where 
$$
\widetilde \Sigma_1=\sum_{t\le X^{1/2}}\min\left\{\frac{X}{t},\, \left(\frac{X}{t}\right)^{1-c}q^{8G_t}\right\}.
$$

For $t \le X^{1/2}$  we have  $\(X/t\)^{1-c\rho_t^2/4} \le   X^{1-c\rho^2/8} t^{-1}$, thus
\begin{align*}
\sum_{t\le X^{1/2}}\left(\frac{X}{t}\right)^{1-c\rho_t^2}&
\le \sum_{t\le X^{1/2}}\left(\frac{X}{t}\right)^{1-c\rho^2/4}\ll   X^{1-c\rho^2/8} \log X.
\end{align*}

This, together with~\eqref{eq:sigma1sigma}, implies 
\begin{equation}
\label{eq:Sigma1step2}
\Sigma_1\ll X^{1-c\rho^2/8} \log X  
+\widetilde \Sigma_1.
\end{equation}

Considering $\widetilde \Sigma_1$,
we  partition summation over $t$ into dyadic intervals to obtain 
\begin{align*}
\widetilde \Sigma_1
&   \ll \sum_{k \leq \frac{\log X}{2\log 2}}\sum_{2^k\leq t < 2^{k+1}}
\min\left\{\frac{X}{t},\, \left(\frac{X}{t}\right)^{1-c}q^{8G_t}\right\} \\
&   \ll \sum_{k \leq \frac{\log X}{2\log 2}}\sum_{2^k\leq t < 2^{k+1}}
\min\left\{\frac{X}{2^k},\, \left(\frac{X}{2^k}\right)^{1-c}q^{8G_t}\right\} .
\end{align*}
Let $k_0$ be such an index with $k_0 \leq (\log X)/(2\log 2)$ that the maximum of the inner sums over $t$ is attained and write 
$$
Z=\frac{X}{2^{k_0}}.
$$
Then
$$X^{1/2}\le Z\le X,$$
and
$$
\widetilde \Sigma_1\ll (\log{X})\sum_{X/Z\le t \le 2X/Z}\min\left\{Z,\, Z^{1-c}q^{8G_t}\right\}.
$$ 
Recalling the definition of $G$, given by~\eqref{eq:def tau G}, we see that 
$$
\ord_q(g^t)=\frac{\tau}{\gcd(\tau,t)},
$$
and by Lemma~\ref{lem:frog2}, used with $m=1$, $x =\tau t/\gcd(\tau,t)$ and $y=0$, 
$$
G_t=\nu_q\(g^{\tau t/\gcd(\tau,t)}-1\)=G+\nu_q(t).
$$

As $g$ and $q$ are  fixed, $G=O(1)$
and hence 
$$
\widetilde \Sigma_1\ll  \log{X} \sum_{X/Z\le t \le 2X/Z}\min\left\{Z,Z^{1-c}q^{8\nu_q(t)}\right\}.
$$ 
For $O(XZ^{-1-c/9})$ values of $t\le 2X/Z$ with $q^{\nu_q(t)}> Z^{c/9}$ we use 
$$\min\left\{Z,Z^{1-c}q^{8\nu_q(t)}\right\}\le Z.$$
Their total contribution is $O\(XZ^{-c/9}\)$. For the remaining values of $t$ we 
$$\min\left\{Z,Z^{1-c}q^{8\nu_q(t)}\right\}\le Z^{1-c+8c/9} = Z^{1-c/9},
$$
which gives the same total contribution $O\(XZ^{-c/9}\)$.
Hence, recalling $Z\ge X^{1/2}$,  we obtain 
$$
\widetilde \Sigma_1 \ll XZ^{-c/9} \log X \le X^{1-c/18} \log X.
$$
 Using the above in~\eqref{eq:Sigma1step2} gives 
\begin{equation}
\label{eq:Sigma1final}
\Sigma_1 \ll X^{1-c\rho^2/8} \log X 
+ X^{1-c/18}  \log{X} 
\ll  X^{1-\delta(A) \rho^2} (\log{X})^2,
\end{equation} 
for some  constant $\delta(A)>0$ that depends only on $A$.  

To estimate $\Sigma_2$ we apply Lemma~\ref{lem:DoubleSumHyperb} to get 

$$
\Sigma_2\ll \left(X^{1-\delta(A) \rho^2}+X^{7/8}+\frac{X}{q^{2\gamma/65}}\right) (\log{X})^2,
$$
for a suitably reduced $\delta(A)$ if necessary.
By the above bounds~\eqref{eq:Vaughan} and~\eqref{eq:Sigma1final}
$$
S_{q^{\gamma}}(a;X) \ll X^{1-\delta(A) \rho^2} (\log{X})^3+\frac{X}{q^{2\gamma/65}}(\log{X})^4.
$$
Now, using the same argument as in the proof of Corollary~\ref{cor:singlesum1}, and reducing $\delta(A)$ 
if necessary, we see that we can replace $(\log{X})^3$ with $\log{X}$ (or any other power of $\log X$) 
in the first term,  and also discard  completely  $(\log{X})^4$ 
in the second term.

\subsection{Proof of Theorem~\ref{thm:MersDig}}  
We observe that the property of having $\sigma$ on positions $r, \ldots, r-s+1$ of $M_p$ 
is equivalent to  the property of the fractional part of $M_p/q^{r+1}$ falling in 
a prescribed half-open interval of length $1/q^s$, namely, to 
\begin{equation}
\label{eq:FracPart}
\left\{\frac{M_p}{q^{r+1}}\right\} \in \left[\frac{\overline \sigma}{q^s},  \frac{\overline \sigma + 1}{q^s}\), 
\end{equation}
(we recall that  the numbering starts from zero) where 
$$
\overline \sigma = \sum_{i=0}^{s-1} a_i q^i
$$
is the integer which $q$-ary digits are given by $\sigma$.  We now combine the bound of 
Corollary~\ref{cor:Mers}  
with the   \emph{Erd{\H o}s--Tur{\'a}n inequality} (see~\cite[Theorem~1.21]{DrTi}), which gives a bound of the discrepancy 
via exponential sums, and conclude that for any integer parameter $H\ge 1$
$$
A_r(X,\sigma)  - q^{-s} \pi(X) 
 \ll  
 \pi(X) H^{-1} + \sum_{h=1}^H \frac{1}{h}
\left|\ssum{p\le X\\p\text{~prime}}\e_{q^{r+1}}(hM_p)\right|.
$$ 
We now set 
$$
H =  \fl{X^{\varepsilon/2} }.
$$
Below we use very crude bounds, many of them can be done in a  more refined way, 
however this does not improve the final result.

Namely, for any positive integer $h \le H$, writing
 $$
 q^\gamma = \frac{q^{r+1}} {\gcd(h, q^{r+1})},
 $$
since $r \ge \varepsilon \log X$,  we see that  
 \begin{equation}
\label{eq:XXv1}
q^{\gamma} \ge q^{r+1}/H \geq e^{r}/H  \ge X^{\varepsilon/2}. 
\end{equation} 

We now use Corollary~\ref{cor:Mers} with $A=2/\varepsilon$   and note by~\eqref{eq:XXv1} the condition~\eqref{eq:mainNcond} is satisfied. This implies that~\eqref{eq:FracPart} happens for 
 \begin{equation}
\label{eq:Penult}
\begin{split}
A_r(X,\sigma) = q^{-s} & \pi(X)\\
& + O\(X^{1-\varepsilon/2} +  X^{1- c  \varrho^2}\log{X} +  X q^{-c r} \log X\)
\end{split} 
\end{equation}  
primes $p \le X$, where 
$$
\varrho = \frac{\log X}{\log 
q^{r+1}}.
$$
and $c>0$ is some constant that depends on $\varepsilon$ and $q$. 

Using that $r \le (\log  X)^{3/2-\varepsilon}$ we obtain  
$\varrho \ge(\log  X)^{-1/2+\varepsilon/2}$.  
Thus 
$$ 
X^{1- c \varrho^2}\log{X} \le X \exp\(- c(\log  X)^{\varepsilon}\)\log{X}.
$$
We also have
$$
 X q^{-c  r}\le X^{1-c \varepsilon}
 $$
and then~\eqref{eq:Penult} implies 
$$
A_r(X,\sigma) = q^{-s}  \pi(X)
+ O\(  X \exp\(- 0.5 c(\log  X)^{\varepsilon}\)\)
$$ 
which concludes the proof.

\section*{Acknowledgement} 
The authors would like to thank Bill Banks for many useful discussions and for his contribution to an early version of the paper. The authors also would like to thank Olivier Bordelles for his interest and very important comments and suggestions. 

The authors are grateful to the anonymous referees for the very careful reading of the manuscript 
and very useful comments.

During the preparation of this work B.K. was supported by ARC Grant~DP160100932, L.M. was supported by the
Austrian Science Fund Project P31762, I.S. was  supported  by ARC Grant DP170100786.


\begin{thebibliography}{9999}

\bibitem{BCFS}
W.~Banks, A.~Conflitti, J.~Friedlander and I.~Shparlinski,
Exponential sums over Mersenne numbers.
\textit{Compos. Math.} 140 (2004),  15--30. 

\bibitem{BFGS}
W.~Banks, J.~Friedlander, M.~Garaev and I.~Shparlinski,
Exponential and character sums with Mersenne numbers.
\textit{J.  Aust. Math. Soc.} 92 (2012),   1--13. 

\bibitem{BDG}
J.~Bourgain, C.~Demeter and L.~Guth,
Proof of the main conjecture in Vinogradov's mean value theorem for degrees higher than three.
\textit{Ann.\ Math.} 184 (2016), 633--682.


\bibitem{CHL}
 Z. Cai,   A. J. Hildebrand and  J. Li, A local Benford law for a class of arithmetic sequences.
  \textit{Int. J. Number Th.}  15  (2019) 613--638.


\bibitem{CFHLZ}
 Z. Cai, M. Faust, A. J. Hildebrand, J. Li and Y. Zhang, Leading digits of Mersenne
numbers. \textit{Exp. Math.} to appear.

\bibitem{Dav}
H.~Davenport,
\textit{Multiplicative number theory}, 2nd edition,  Springer-Verlag, New York-Berlin, 1980.  


\bibitem{DrTi} M. Drmota and R. F. Tichy,
\textit{Sequences, discrepancies and
applications}, Springer-Verlag, Berlin, 1997.  

\bibitem{Ford}
K.~Ford,
Vinogradov's integral and bounds for the Riemann zeta function. 
\textit{Proc. London Math.\ Soc.} 85 (2002), 565--633.



\bibitem{Gar2}
M.~Garaev,
An estimate of Kloosterman sums with prime numbers and an application.
\textit{Matem. Zametki} 88 (2010), 365--373 (in Russian).



\bibitem{HHLZ}
 X. He, A. J. Hildebrand, J. Li and Y. Zhang, Complexity of leading digit sequences. 
 \textit{Discrete Math.  Theor. Comp. Sci.}  22 (2020),   no.~1, Article \# 14, 1--30.



\bibitem{IwKow} H. Iwaniec and E. Kowalski,
 \textit{Analytic Number Theory}, Amer.  Math.  Soc.,
Providence, RI, 2004.

\bibitem{Kor}
N.~Korobov, On the distribution of digits in periodic fractions.
\textit{Mathem. USSR Sbornik} 18 (1972), 659--676,
 (translated from \textit{Matem.  Sbornik} 89 (1972), 654--670).
 
 \bibitem{Ste} R.  S. Steiner, 
Effective Vinogradov's mean value theorem via efficient boxing.
 \textit{J. Number Theory} 204 (2019),  354--404.


Author links open overlay panel
\bibitem{Wool}
T.~Wooley,
The cubic case of the main conjecture in Vinogradov's mean value theorem.
\textit{Adv.  in Math.} 294 (2016), 532--561.

\end{thebibliography}
\end{document}